\documentclass{amsart}
\usepackage{graphicx}
\usepackage{latexsym}
\usepackage{hyperref}
\usepackage{amsfonts}
\usepackage[all]{xy}
\usepackage{amssymb, mathrsfs, amsfonts, amsmath, mathtools}
\usepackage{amsbsy}
\usepackage{amsfonts}
\newtheorem{definition}{Definition}
\newtheorem{lemma}{Lemma}
\newtheorem{proposition}{Proposition}
\newtheorem{remark}{Remark}
\newtheorem{theorem}{Theorem}

\numberwithin{equation}{section}

\begin{document}
	
	\title[Notes on a Cotton-type tensor for electrostatic systems]{Notes on a Cotton-type tensor for electrostatic systems}

	\author{R\'obson Lousa}
	\address{Federal Institute of Goiás, Campus Uruaçu, Uruaçu, GO, Brazil.}
	\curraddr{}
	\email{robson.lousa@ifg.edu.br}
	
	
	\date{}
	
	\dedicatory{}
	
	\begin{abstract}
		We introduce a natural $(0,3)$-tensor associated with electrostatic systems in arbitrary dimensions. This tensor arises from the comparison between the Cotton and Weyl decompositions of the Riemann curvature tensor and extends several identities previously known in dimensions three and four. We prove that it is totally trace-free and satisfies the same algebraic symmetries as the Cotton tensor. We further investigate its behavior under the natural assumption that the electric field is collinear with the gradient of the lapse function, obtaining a simplified expression. Along the way, we extend a boundary collinearity result to arbitrary dimensions and derive several identities that may be useful in the study of electrostatic systems.
	\end{abstract}
	
	\maketitle
	
	\section{Introduction and main results}
	
	Electrostatic systems have recently attracted considerable attention in Differential Geometry and Mathematical Relativity. They naturally arise as the spatial part of static solutions to the Einstein--Maxwell equations and provide a rich interaction between geometric analysis, curvature identities and rigidity phenomena \cite{leandro2023geometry,leandro2024electrostatic,cederbaum2016uniqueness,chrusciel2005non,chrusciel2017non,tiarlos}.
	
	Two equivalent formulations of the electrostatic system appear in the literature. The first, considered in \cite{leandro2023geometry}, is expressed in terms of the pair $(f,\psi)$, where $f$ denotes the lapse function and $\psi$ is the electric potential. In this setting, the authors obtained several classification and rigidity results, both in arbitrary dimensions and in dimension three. A second formulation, adopted in \cite{leandro2024electrostatic,tiarlos}, is written in terms of the lapse function $f$ and the electric field $E$, leading to the quadruple $(M^n,g,f,E)$.
	
	These two descriptions are equivalent. Indeed, if $f>0$ on $M$, then one simply defines
	\[
	E=\frac{\nabla\psi}{f},
	\]
and the corresponding equations are equivalent. Throughout this paper we adopt the formulation in terms of the electric field.

\begin{definition}\label{def1}
	Let $(M^n,g)$ be a Riemannian manifold and let $E$ be a vector field on $M$ (the electric field) and $f\in C^\infty(M)$ (the lapse function). We say that $(M^n,g,f,E)$ satisfies the \emph{electrostatic system with cosmological constant $\Lambda$} if
	\begin{equation}\label{s1}
		\begin{array}{rcl}
			\nabla^2f&=&f\left(\mathrm{Ric}-\dfrac{2}{n-1}\Lambda g+2E^\flat\otimes E^\flat-\dfrac{2}{n-1}|E|^2g\right),\\\\
			\Delta f&=&2f\left(\dfrac{n-2}{n-1}|E|^2-\dfrac{\Lambda}{n-1}\right),\\[0.2cm]
			\mathrm{div}(E)&=&0,\qquad d(fE^\flat)=0.
		\end{array}
	\end{equation}
	Here $\mathrm{Ric}$, $\nabla^2$, $\mathrm{div}$ and $\Delta$ denote, respectively, the Ricci tensor, Hessian, divergence and Laplacian with respect to $g$, and $E^\flat$ is the $g$-dual one-form associated with $E$.	The condition $d(fE^\flat)=0$ is the coordinate-free formulation of the Maxwell equation and corresponds to the vanishing of the curl of $fE$ in the three-dimensional setting.  The Riemannian manifold $(M^n,g)$ is called the \emph{spatial factor} of the corresponding static electrostatic spacetime.
\end{definition}
	
Throughout the paper we shall use the following standard facts (see, for instance,
\cite{leandro2023geometry,leandro2024electrostatic,chrusciel2005non,chrusciel2017non,tiarlos}):

\begin{enumerate}
	\item[$\circ$] $f>0$ on $M$;
	\item[$\circ$] if $M$ has nonempty boundary $\Sigma$, then
	\[
	\Sigma=f^{-1}(0).
	\]
\end{enumerate}

The first goal of this paper is to extend a boundary property of the electric field established in dimension three by Cruz, Lima and Sousa \cite{tiarlos}.

\begin{lemma}\label{lemmafronteira}
	Let $(M^n,g,f,E)$ be an $n$-dimensional electrostatic system, $n\ge2$. Suppose that
	\[
	\Sigma=f^{-1}(0)
	\]
	is a regular level set of $f$. Then $E$ is normal to $\Sigma$. More precisely, there exists a smooth function $\rho$ on $\Sigma$ such that
	\[
	E=\rho\nabla f
	\qquad\text{along }\Sigma.
	\]
\end{lemma}

The above lemma shows that the electric field is necessarily normal to the horizon. This boundary behavior naturally motivates the study of electrostatic systems under the assumption that the electric field remains collinear with the gradient of the lapse function throughout the manifold.

The main objective of this paper is to introduce a natural $(0,3)$-tensor associated with electrostatic systems. The introduction of auxiliary tensors has proved to be a fruitful approach in the study of several geometric structures. A remarkable example is the $D$-tensor introduced by Cao and Chen \cite{cao} for gradient Ricci solitons, which has played a central role in rigidity and classification results. Similar ideas have subsequently been employed in the study of gradient Yamabe solitons \cite{sun}, generalized quasi-Einstein manifolds \cite{catino2012note,deng2015note,neto2016generalized}, Einstein-type manifolds \cite{leandro2021}, and manifolds satisfying curvature conditions involving the Weyl tensor \cite{catino2017gradient,catino-mastrolia-monticelli-punzo}.

Motivated by these developments, we introduce a new tensor $V$ for electrostatic systems. Our construction extends the tensor obtained by Leandro, Andrade and Lousa \cite{leandro2023geometry} for the formulation $(M^n,g,f,\psi)$ and the tensor introduced by Leandro and Lousa \cite{leandro2024electrostatic} in dimension three. Although inspired by the $D$-tensor, the tensor $V$ reflects the additional geometric structure carried by the electromagnetic field and therefore contains new terms that are absent in the purely Riemannian setting. As we shall prove, $V$ naturally relates the Cotton and Weyl tensors through the identity
\[
fC=W(\nabla f)+V,
\]
and provides a convenient framework for deriving new geometric identities in arbitrary dimensions.

\begin{theorem}\label{principal}
	Let $(M^n,g,f,E)$ be an electrostatic system. Then
	\[
	fC=W(\nabla f)+V,
	\]
	where $C$ and $W$ denote the Cotton and Weyl tensors, respectively.
\end{theorem}

We also prove that the tensor $V$ is totally trace-free and satisfies the same algebraic symmetries as the Cotton tensor.

In the final section we investigate the tensor $V$ under the natural collinearity assumption
\[
E=\rho\nabla f.
\]
Unlike the three-dimensional case considered in \cite{leandro2024electrostatic}, we derive an explicit formula for $\nabla\rho$, which enables us to express the tensor $V$ solely in terms of the geometric data of the electrostatic system. We also extend to arbitrary dimensions the function $Q$ introduced in \cite{leandro2024electrostatic} and establish a density result for the set $\{Q\neq0\}$.

\begin{lemma}\label{lemmabom}
	The set
	\[
	\{Q\neq0\}
	\]
	is dense in every electrostatic system satisfying
	\[
	E=\rho\nabla f.
	\]
\end{lemma}

The paper is organized as follows. In Section~2 we collect some basic identities for electrostatic systems. In Section~3 we introduce the tensor $V$ and establish its fundamental algebraic properties. Section~4 is devoted to the study of $V$ under the collinearity assumption $E=\rho\nabla f$, where we also prove a density result for the generalized function $Q$.
	
\section{Preliminaries}

In this section we collect some basic identities satisfied by electrostatic
systems that will be repeatedly used throughout the paper. Most of these
relations are direct consequences of the electrostatic system
\eqref{s1} and are well known in the literature
\cite{leandro2023geometry,leandro2024electrostatic,tiarlos}.

Contracting the first equation of \eqref{s1} and using the Laplace equation (see \cite{leandro2024electrostatic,tiarlos}), we obtain .
	
\begin{equation}\label{rrr}
	R=2(|E|^2+\Lambda).
\end{equation}

Observe that this identity is independent of the dimension $n$ and expresses
the scalar curvature entirely in terms of the electric field and the
cosmological constant.
	
Substituting \eqref{rrr} into the first equation of \eqref{s1}, we obtain the
equivalent formulation

\begin{equation}\label{combinado}
	\nabla^2f=f\left(\mathrm{Ric}+2E^\flat\otimes E^\flat -\dfrac{R}{n-1}g\right).
\end{equation}

The Maxwell equation also admits the following coordinate expression, which
will be frequently used in the sequel. We define $E^\flat_i:=E^\flat(e_i)$
\begin{lemma}\label{lemmaMaxwell}
	Let $(M^n,g,f,E)$ be an electrostatic system. Then
	\[
	f(\nabla_iE^\flat_j-\nabla_jE^\flat_i)
	=
	\nabla_jfE^\flat_i-\nabla_ifE^\flat_j.
	\]
\end{lemma}

\begin{proof}
	The identity follows immediately from the equation
	\[
	d(fE^\flat)=0.
	\]
\end{proof}

The electrostatic system also satisfies $\mathrm{div}(E)=0$ that is
\[
\nabla_iE^i=0,
\]
which will play an important role under the collinearity assumption considered
later.

\section{The electrostatic tensor}\label{tensor}
Taking the covariant derivative of \eqref{combinado} and using the Ricci identity, we obtain
\begin{align}
	R_{ijkl}\nabla^lf
	={}&\,\nabla_i\nabla_j\nabla_kf
	-\nabla_j\nabla_i\nabla_kf\notag\\
	={}&
	(\nabla_if\,R_{jk}
	-\nabla_jf\,R_{ik}) \notag\\
	&
	-\frac{R}{n-1}
	\left(
	\nabla_if\,g_{jk}
	-
	\nabla_jf\,g_{ik}
	\right) \notag\\
	&
	+f(\nabla_iR_{jk}-\nabla_jR_{ik}) \notag\\
	&
	+2f\left(
	E^\flat_j\nabla_iE^\flat_k
	-
	E^\flat_i\nabla_jE^\flat_k
	\right)\notag\\
	&
	-\frac{f}{n-1}
	\left(
	\nabla_iR\,g_{jk}
	-
	\nabla_jR\,g_{ik}
	\right).
\end{align}
The Cotton tensor is given by
\begin{equation}
	C_{ijk}
	=
	\nabla_iR_{jk}
	-
	\nabla_jR_{ik}
	-
	\frac{1}{2(n-1)}
	\left(
	\nabla_iR\,g_{jk}
	-
	\nabla_jR\,g_{ik}
	\right).
\end{equation}

So 

\begin{align}
	R_{ijkl}\nabla^lf
	={}&\,fC_{ijk}+
	(\nabla_if\,R_{jk}
	-\nabla_jf\,R_{ik}) \notag\\
	&
	-\frac{R}{n-1}
	\left(
	\nabla_if\,g_{jk}
	-
	\nabla_jf\,g_{ik}
	\right) \notag\\
	&
	-\frac{f}{2(n-1)}
	\left(
	\nabla_iR\,g_{jk}
	-
	\nabla_jR\,g_{ik}
	\right) \notag\\
	&
	+2f\left(
	E^\flat_j\nabla_iE^\flat_k
	-
	E^\flat_i\nabla_jE^\flat_k
	\right).
\end{align}

Weyl tensor 

\begin{equation}
	\begin{aligned}
		R_{ijkl}
		={}&\,W_{ijkl}
		+\frac{1}{n-2}
		\left(
		R_{ik}g_{jl}
		-R_{il}g_{jk}
		+R_{jl}g_{ik}
		-R_{jk}g_{il}
		\right)\\
		&
		-\frac{R}{(n-1)(n-2)}
		\left(
		g_{ik}g_{jl}
		-g_{il}g_{jk}
		\right).
	\end{aligned}
\end{equation}
Then
\begin{align}
	R_{ijkl}\nabla^l f
	={}&\,W_{ijkl}\nabla^l f \notag\\
	&
	+\frac{1}{n-2}
	\left(
	R_{ik}\nabla_jf
	-R_{il}\nabla^lf\,g_{jk}
	+R_{jl}\nabla^lf\,g_{ik}
	-R_{jk}\nabla_if
	\right)\notag\\
	&
	-\frac{R}{(n-1)(n-2)}
	\left(
	\nabla_jf\,g_{ik}
	-\nabla_if\,g_{jk}
	\right).
\end{align}

Comparing,
\begin{align}\label{final}
	fC_{ijk}
	={}&
	W_{ijkl}\nabla^lf + \frac{n-1}{n-2}
	\left(
	R_{ik}\nabla_jf
	-
	R_{jk}\nabla_if
	\right)\notag\\
	&
	+\frac{1}{n-2}
	\left(
	R_{jl}\nabla^lf\,g_{ik}
	-
	R_{il}\nabla^lf\,g_{jk}
	\right)\notag\\
	&
	+\frac{R}{n-2}
	\left(
	\nabla_if\,g_{jk}
	-
	\nabla_jf\,g_{ik}
	\right)\\
	&
	+\frac{f}{2(n-1)}
	\left(
	\nabla_iR\,g_{jk}
	-
	\nabla_jR\,g_{ik}
	\right)\notag\\
	&
	-2f
	\left(
	E^\flat_j\nabla_iE^\flat_k
	-
	E^\flat_i\nabla_jE^\flat_k
	\right).\notag
\end{align}

Motivated by the previous identity, and inspired by the $D$-tensor introduced by Cao and Chen \cite{cao} in the study of gradient Ricci solitons, we introduce the following $(0,3)$-tensor.

\begin{definition}[The electrostatic tensor]\label{defV}
	Let $(M^n,g,f,E)$ be an $n$-dimensional electrostatic system. We define the \emph{electrostatic tensor} $V$ as the $(0,3)$-tensor given by
	\begin{align}
		V_{ijk}
		={}&\frac{n-1}{n-2}
		\left(
		R_{ik}\nabla_jf
		-
		R_{jk}\nabla_if
		\right)\notag\\
		&
		+\frac{1}{n-2}
		\left(
		R_{jl}\nabla^lf\,g_{ik}
		-
		R_{il}\nabla^lf\,g_{jk}
		\right)\notag\\
		&
		+\frac{R}{n-2}
		\left(
		\nabla_if\,g_{jk}
		-
		\nabla_jf\,g_{ik}
		\right)\\
		&
		+\frac{f}{2(n-1)}
		\left(
		\nabla_iR\,g_{jk}
		-
		\nabla_jR\,g_{ik}
		\right)\notag\\
		&
		-2f
		\left(
		E^\flat_j\nabla_iE^\flat_k
		-
		E^\flat_i\nabla_jE^\flat_k
		\right).\notag
	\end{align}
\end{definition}

\begin{proposition}\label{prop:trace}
	The electrostatic tensor is totally trace-free.
\end{proposition}

\begin{proof}
	A direct computation shows that
	\[
	g^{ij}V_{ijk}
	=
	g^{ik}V_{ijk}
	=
	g^{jk}V_{ijk}
	=
	0.
	\]
	This completes the proof.
\end{proof}

\begin{proposition}\label{prop:symmetry}
	The electrostatic tensor has the same algebraic symmetries as the Cotton tensor.
\end{proposition}

\begin{proof}
	A direct computation yields
	\[
	V_{ijk}
	=
	-
	V_{jik},
	\]
	and
	\[
	V_{ijk}
	+
	V_{jki}
	+
	V_{kij}
	=
	0.
	\]
	Therefore, $V$ satisfies the same algebraic symmetries as the Cotton tensor.
\end{proof}

The main result of this section is the following identity relating the electrostatic tensor, the Cotton tensor and the Weyl tensor.

\begin{theorem}[Theorem \ref{principal}]
	Let $(M^n,g,f,E)$ be an electrostatic system. Then
	\[
	fC=W(\nabla f)+V,
	\]
	where $C$ and $W$ denote the Cotton and Weyl tensors, respectively.
\end{theorem}
\begin{proof}
	The identity is exactly equation \eqref{final}.
\end{proof}

The previous propositions show that the tensor $V$ shares the fundamental algebraic properties of the Cotton tensor, reinforcing its role as a natural geometric object associated with electrostatic systems.


\section{The collinearity assumption}\label{colineary}

In this section we investigate the electrostatic tensor $V$ under the additional assumption that the electric field is everywhere collinear with the gradient of the lapse function. More precisely, we assume that there exists a smooth function $\rho$ on $M$ such that
\[
E=\rho\nabla f.
\]

This assumption naturally simplifies the structure of the tensor $V$ and leads to a considerably more explicit expression. Moreover, as shown below, it is automatically satisfied along the boundary of every electrostatic system.

The following lemma extends the boundary collinearity result of Cruz, Lima and Sousa \cite{tiarlos} to arbitrary dimensions. Besides being of independent interest, it provides the motivation for studying the tensor $V$ under the global collinearity assumption.

\begin{lemma}[Lemma \ref{lemmafronteira}]
	Let $(M^n,g,f,E)$ be an $n$-dimensional electrostatic system, $n\ge2$. Suppose that
	\[
	\Sigma=f^{-1}(0)
	\]
	is a regular level set of $f$. Then $E$ is normal to $\Sigma$. More precisely, there exists a smooth function $\rho$ on $\Sigma$ such that
	\[
	E=\rho\nabla f
	\qquad\text{along }\Sigma.
	\]
\end{lemma}

\begin{proof}
	Since
	\[
	d(fE^\flat)=df\wedge E^\flat+f\,dE^\flat=0,
	\]
	and $f=0$ along $\Sigma$, it follows that
	\[
	df\wedge E^\flat=0
	\qquad\text{on }\Sigma.
	\]
	
	Let $p\in\Sigma$. Since $\Sigma$ is a regular level set,
	\[
	\nabla f(p)\neq0,
	\]
	and therefore
	\[
	T_p\Sigma=\ker(df_p).
	\]
	
	If $X\in T_p\Sigma$, then $df(X)=0$, and hence
	\[
	\begin{aligned}
		0
		&=(df\wedge E^\flat)(\nabla f,X)\\
		&=df(\nabla f)E^\flat(X)-df(X)E^\flat(\nabla f)\\
		&=|\nabla f|^2E^\flat(X).
	\end{aligned}
	\]
	
	Since $|\nabla f|^2>0$ at $p$, we conclude that
	\[
	E^\flat(X)=0
	\qquad\forall\,X\in T_p\Sigma.
	\]
	
	Thus $E^\flat$ annihilates the tangent space of $\Sigma$, so it is proportional to the conormal $df$. Therefore, there exists a smooth function $\rho$ on $\Sigma$ such that
	\[
	E^\flat=\rho\,df
	\qquad\text{along }\Sigma,
	\]
	or equivalently,
	\[
	E=\rho\nabla f
	\qquad\text{along }\Sigma.
	\]
\end{proof}

From now on, we assume that
\[
E=\rho\nabla f
\]
throughout $M$ to prove the following lemmas.

\begin{lemma}\label{lemmaE}
	Under the collinearity assumption
	\[
	E=\rho\nabla f,
	\]
	the covariant derivative of the electric field satisfies
	\[
	\nabla_iE^\flat_k
	=
	\nabla_i\rho\,\nabla_kf
	+f\rho R_{ik}
	+2f\rho^3\nabla_if\nabla_kf
	-\frac{f\rho R}{n-1}g_{ik}.
	\]
\end{lemma}

\begin{proof}
	Differentiating 
	\[
		E=\rho\nabla f
	\] and using \eqref{combinado}, we obtain
	\[
	\begin{aligned}
		\nabla_iE^\flat_k
		&=\nabla_i(\rho\nabla_kf)\\
		&=\nabla_i\rho\,\nabla_kf
		+\rho\nabla_i\nabla_kf\\
		&=\nabla_i\rho\,\nabla_kf
		+\rho f\left(
		R_{ik}
		+2\rho^2\nabla_if\nabla_kf
		-\frac{R}{n-1}g_{ik}
		\right)\\
		&=\nabla_i\rho\,\nabla_kf
		+f\rho R_{ik}
		+2f\rho^3\nabla_if\nabla_kf
		-\frac{f\rho R}{n-1}g_{ik}.
	\end{aligned}
	\]
	
\end{proof}

Next, we derive an explicit expression for the gradient of the scalar curvature.

\begin{lemma}\label{lemmaR}
	Under the collinearity assumption,
	\[
	\nabla_iR
	=
	4\rho|\nabla f|^2\nabla_i\rho
	+
	4f\rho^2R_{il}\nabla^lf
	+
	8f\rho^4|\nabla f|^2\nabla_if
	-
	\frac{4f\rho^2R}{n-1}\nabla_if.
	\]
\end{lemma}

\begin{proof}
	Differentiating equation \eqref{rrr} and using the relation
	\[
	E=\rho\nabla f,
	\]
	we obtain
	\[
	\nabla_i R=2\nabla_i|E|^2
	=4\rho|\nabla f|^2\nabla_i\rho+2\rho^2\nabla_i|\nabla f|^2.
	\]
	From \eqref{combinado},
	\[
	\nabla_i|\nabla f|^2
	=2f\left(R_{il}\nabla^l f+2\rho^2|\nabla f|^2\nabla_i f-\frac{R}{n-1}\nabla_i f\right).
	\]
	Combining these expressions and using \eqref{rrr}, we find
	
	\[
	\begin{aligned}
		\nabla_iR
		={}&\,4\rho|\nabla f|^2\nabla_i\rho
		+4f\rho^2R_{il}\nabla^lf+8f\rho^4|\nabla f|^2\nabla_if
		-\frac{4f\rho^2R}{n-1}\nabla_if.
	\end{aligned}
	\]
\end{proof}

The next lemma provides an explicit formula for the gradient of the function $\rho$, which plays a key role in the simplification of the electrostatic tensor.

\begin{lemma}\label{lemmarho}
	Under the collinearity assumption,
	\[
	|\,\nabla f\,|^2\nabla\rho
	=
	-\frac{2\rho f}{n-1}
	\left(
	(n-1)\rho^2|\nabla f|^2
	-\frac{R}{2}
	\right)
	\nabla f.
	\]
\end{lemma}

\begin{proof}
	Since
	\[
	\operatorname{div}E=0
	\]
	and
	\[
	E=\rho\nabla f,
	\]
	we have
	\[
	0=\operatorname{div}(\rho\nabla f)
	=\nabla_i(\rho\nabla^if)
	=\nabla_i\rho\,\nabla^if+\rho\,\Delta f.
	\]
	Hence,
	\[
	\langle\nabla\rho,\nabla f\rangle
	=-\rho\,\Delta f.
	\]
	
	Moreover, since
	\[
	d(fE^\flat)=0,
	\]
	it follows that
	\[
	\nabla_i\rho\,\nabla_jf
	-
	\nabla_j\rho\,\nabla_if=0,
	\]
	so that \(\nabla\rho\) is parallel to \(\nabla f\). Therefore, there exists a smooth function \(\mu\) such that
	\[
	\nabla\rho=\mu\nabla f.
	\]
	Taking the inner product with \(\nabla f\), we obtain
	\[
	\mu|\nabla f|^2
	=
	-\rho\,\Delta f,
	\]
	that is,
	\[
	|\nabla f|^2\nabla\rho
	=
	-\rho\,\Delta f\nabla f.
	\]
	Using
	\[
	\Delta f
	=
	2f\left(
	\frac{n-2}{n-1}\rho^2|\nabla f|^2
	-
	\frac{\Lambda}{n-1}
	\right),
	\]
	we get that
	\[
	|\nabla f|^2\nabla\rho
	=
	-\frac{2\rho f}{(n-1)}
	\left(
	(n-2)\rho^2|\nabla f|^2-\Lambda
	\right)\nabla f,
	\]

	Using \eqref{rrr} together with the identity $E=\rho\nabla f$, we conclude
	
	\[
	|\nabla f|^2\nabla\rho
	=
	-\frac{2\rho f}{(n-1)}
	\left(
	(n-1)\rho^2|\nabla f|^2-\dfrac{R}{2}
	\right)\nabla f.
	\]
\end{proof}

\begin{proposition}\label{propv}
	Under the collinearity assumption, the electrostatic tensor can be written as

\begin{align}
	V_{ijk}
	={}&
	\left(
	\frac{n-1-2(n-2)f^2\rho^2}{n-2}
	\right)
	\left(
	R_{ik}\nabla_jf
	-
	R_{jk}\nabla_if
	\right)\notag\\
	&
	-\left(
	\frac{n-1-2(n-2)f^2\rho^2}{(n-1)(n-2)}
	\right)
	\left(
	R_{il}\nabla^lf\,g_{jk}
	-
	R_{jl}\nabla^lf\,g_{ik}
	\right)\notag\\
	&
	+\left(
	\frac{n-1-2(n-2)f^2\rho^2}{(n-1)(n-2)}R\right)
	\left(
	\nabla_if\,g_{jk}
	-
	\nabla_jf\,g_{ik}
	\right)\notag
\end{align}
\end{proposition}
\begin{proof}
	Substituting the identities established in Lemmas~\ref{lemmaE}--\ref{lemmarho} into Definition \ref{defV} and simplifying the resulting expression, we obtain the desired formula.
\end{proof}

The coefficient
\[
n-1-2(n-2)f^2\rho^2
\]
appears naturally in the simplified expression of the electrostatic tensor.
Motivated by this observation, we introduce the following scalar function, which
is the natural higher-dimensional counterpart of the one introduced in
\cite{leandro2024electrostatic}.

\begin{definition}\label{defQ}
	Under the collinearity assumption, we define the scalar function
	\[
	Q:=n-1-2(n-2)f^2\rho^2.
	\]
\end{definition}

\begin{lemma}\label{lemmaQ}
	The set
	\[
	\{Q\neq0\}
	\]
	is dense in every electrostatic system $(M^n,g,f,E)$ satisfying
	\[
	E=\rho\nabla f.
	\]
\end{lemma}
\begin{proof}
	Assume, for contradiction, that there exists a nonempty open subset
	$\Omega\subset M$ on which
	\[
	n-1-2(n-2)f^2\rho^2=0.
	\]
	
	Differentiating the above identity, we obtain
	\[
	f\rho^2\nabla f+f^2\rho\nabla\rho=0.
	\]
	Taking the inner product with $\nabla f$ and using
	$|E|^2=\rho^2|\nabla f|^2$, we obtain
	\[
	|E|^2+f\rho\langle\nabla\rho,\nabla f\rangle=0.
	\]
	
	On the other hand, by \eqref{s1},
	\[
	0=\operatorname{div}(E)
	=\rho\Delta f+\langle\nabla\rho,\nabla f\rangle
	=\frac{2}{n-1}\rho f[(n-1)2|E|^2-\Lambda]
	+\langle\nabla\rho,\nabla f\rangle.
	\]
	Hence,
	\[
	f\rho\langle\nabla\rho,\nabla f\rangle
	=\frac{2}{n-1}\rho^2f^2[\Lambda-(n-2)|E|^2]
	=\frac{1}{n-2}[\Lambda-(n-2)|E|^2].
	\]
	
	Comparing the last two identities, we obtain
	\[
	|E|^2+\frac{1}{n-2}[\Lambda-(n-2)|E|^2]=0,
	\]
	which implies
	\[
	\Lambda=0.
	\]
	This contradicts the standing assumption that $\Lambda\neq0$.
	Therefore, the set
	\[
	\{n-1-2(n-2)f^2\rho^2\neq0\}
	\]
	is dense in \(M\).
\end{proof}

\begin{remark}
	The explicit expression obtained in Proposition\ref{propv} suggests that the tensor $V$
	admits a remarkably simpler structure under the collinearity assumption. The
	function
	\[
	Q=n-1-2(n-2)f^2\rho^2
	\]
	naturally appears as the coefficient of the Ricci terms in the expression of
	$V$. It should be viewed as the higher-dimensional analogue of the function
	introduced in \cite{leandro2024electrostatic}, where it plays a similar role in the three-dimensional theory.
\end{remark}

\section{Final remarks}

The tensor introduced in this paper provides a natural counterpart of the D-tensor for electrostatic systems. Besides its algebraic properties, the identities established here suggest that V may play an important role in rigidity questions for electrostatic systems.

Future work will investigate the norm of V, divergence identities and possible rigidity results arising from the interaction between V, the Cotton tensor and the Weyl tensor.

	\bibliographystyle{amsplain}

\end{document}